\documentclass[12pt]{article}

\usepackage{graphics,amsmath,amssymb,amsthm,mathrsfs}

\newcommand{\varep}{\varepsilon}

\newtheorem{thm}{Theorem}[section]
\newtheorem{lemma}[thm]{Lemma}

\newtheorem{remark}[thm]{Remark}

\numberwithin{equation}{section}

\begin{document}

\bibliographystyle{amsplain}

\title{Commutator Estimates  for \\ the Dirichlet-to-Neumann Map  in Lipschitz Domains
}

\author{Zhongwei Shen\thanks{Supported in part by NSF grant DMS-1161154}}

\date{ }

\maketitle

\begin{abstract}

We establish two commutator estimates for the Dirichlet-to-Neumann map associated with  a second-order
elliptic system in divergence form in Lipschitz domains.
Our approach is based on Dahlberg's bilinear estimates.

\medskip


\end{abstract}


\section{Introduction}

This paper concerns the $L^2$ boundedness of the commutator
\begin{equation}\label{commutator}
[\Lambda, g ]f=\Lambda (gf) -g \Lambda (f),
\end{equation}
where $\Lambda$ is the Dirichlet-to-Neumann map associated with the
elliptic operator
\begin{equation}\label{L}
\mathcal{L}=-\text{div} \big( A(x)\nabla \big)
=-\frac{\partial}{\partial x_i} \left\{ a_{ij}^{\alpha\beta} (x)
\frac{\partial}{\partial x_j} \right\} \quad \text{ in } \Omega,
\end{equation}
and $\Omega$ is  a bounded Lipschitz domain in $\mathbb{R}^d$.
Assume that $A=A(x) = \big( a_{ij}^{\alpha\beta} (x) \big)$, with  $1\le i, j\le d$ and $1\le \alpha, \beta\le m$, is real,
bounded measurable, and
satisfies the ellipticity condition:
\begin{equation}\label{ellipticity}
\mu |\xi|^2\le a_{ij}^{\alpha\beta} (x)\xi_i^\alpha\xi_j^\beta
 \le \frac{1}{\mu} |\xi|^2, \quad\text{ for any } \xi=(\xi_i^\alpha)\in \mathbb{R}^{d\times m}
 \text{ and } x\in \mathbb{R}^d,
\end{equation}
where $\mu>0$. It is  well known that for any $f=(f^\alpha)\in H^{1/2}(\partial\Omega; \mathbb{R}^m)$, 
 the Dirichlet problem
\begin{equation}\label{Dirichlet}
\mathcal{L}(u) =0 \quad \text{ in } \Omega \quad \text{ and } \quad 
u= f \quad \text{ on } \partial\Omega
\end{equation}
has a unique solution in $H^1(\Omega; \mathbb{R}^m)$.
The Dirichlet-to-Neumann map $\Lambda$ is defined by
\begin{equation}\label{D-to-N}
\big(\Lambda (f) \big)^\alpha =\left(\frac{\partial u}{\partial \nu}\right)^\alpha
=n_i a_{ij}^{\alpha\beta}  \frac{\partial u^\beta}{\partial x_j}.
\end{equation}
This gives a bounded linear map $\Lambda : H^{1/2} (\partial\Omega; \mathbb{R}^m)
\to H^{-1/2} (\partial\Omega;\mathbb{R}^m)$.
Under the additional symmetry condition $A^*=A$, i.e., $a_{ij}^{\alpha\beta}
=a_{ji}^{\beta\alpha}$, as well as some smoothness condition on $A$,
one may show that $\Lambda : H^1(\partial\Omega; \mathbb{R}^m) \to L^2(\partial\Omega; \mathbb{R}^m)$
is bounded (see \cite {Dahlberg-1979, Jerison-Kenig-1980, Verchota-1984, Dahlberg-Kenig-1987, 
Mitrea-Taylor-2000, Kenig-Shen-2}
and their references). The following are our main results of the paper.

\begin{thm}\label{main-theorem-1}
Let $\Omega$ be a bounded Lipschitz domain in $\mathbb{R}^d$.
Assume that $A$ satisfies the ellipticity condition (\ref{ellipticity}) and
the symmetry condition $A^*=A$. Also assume that $A$ is H\"older continuous.
Then, for any $\mathbb{R}^m$-valued Lipschitz function $f$ and
any scalar Lipschitz function $g$ on $\partial\Omega$,
\begin{equation}\label{1.1}
\|\Lambda (gf)-g\Lambda (f) \|_{L^2 (\partial\Omega)}
\le C \, \| g\|_{C^{0,1}(\partial\Omega)} \| f\|_{L^2(\partial\Omega)},
\end{equation}
where $C$ depends only on $A$ and $\Omega$.
\end{thm}

\begin{thm}\label{main-theorem-2}
Assume that $\Omega$ and $A$ satisfy the same conditions as in Theorem \ref{main-theorem-1}.
Then, for any $\mathbb{R}^m$-valued Lipschitz function $f$ and
any scalar Lipschitz function  $g$ on $\partial\Omega$,
\begin{equation}\label{1.2}
\|\Lambda (gf)- g \Lambda (f) \|_{L^2(\partial\Omega)}
\le C\,  \| u\|_{L^\infty (\Omega)} \| g\|_{H^1(\partial\Omega)},
\end{equation}
where $u$ is the solution of (\ref{Dirichlet}) with boundary data $f$.
The constant $C$ in (\ref{1.2}) depends only on $A$ and $\Omega$.
\end{thm}

A few remarks are in order.

\begin{remark}\label{remark-1.1}
{\rm
If $\partial\Omega$ and $A$ are sufficiently smooth, the operator $\Lambda$ is a pseudo-differential operator
of order one. It follows that
$$
\|\Lambda (gf)-g \Lambda (f)\|_{L^p(\partial\Omega)} 
\le C \| g\|_{C^{0,1}(\partial\Omega)} \| f\|_{L^p(\partial\Omega)}
$$
for any $1<p<\infty$. This is a classical commutator estimate due to A.P. Calder\'on \cite{Calderon-1967}.
}
\end{remark}

\begin{remark}\label{remark-1.2}
{\rm 
If $m=1$, or $m>1$ and $d=2, 3$, the solutions of the Dirichlet problem (\ref{Dirichlet})
in Lipschitz domains
satisfy the Agmon-Miranda  maximum principle,
\begin{equation}\label{max}
\| u\|_{L^\infty(\Omega)} \le C\, \| f\|_{L^\infty(\partial\Omega)}
\end{equation}
(it is the usual maximum principle for $m=1$; see \cite{Dahlberg-Kenig-1990} for $m>1$ and $d=2,3$).
In this case the estimate (\ref{1.2}) may be replaced by
\begin{equation}\label{1.2-1}
\|\Lambda (gf)- g \Lambda (f) \|_{L^2(\partial\Omega)}
\le C\, \| f\|_{L^\infty (\partial\Omega)} \| g\|_{H^1(\partial\Omega)},
\end{equation}
which, by duality, is equivalent to
\begin{equation}\label{1.2-2}
\|\Lambda (gf)-g \Lambda (f)\|_{L^1(\partial\Omega)} \le C \, \| f\|_{L^2(\partial\Omega)}\| g\|_{H^1(\partial\Omega)}.
\end{equation}
We mention that  (\ref{max}) also holds in $C^1$ domains for $m>1$ and $d\ge 2$.
Whether the estimate (\ref{max}) holds in Lipschitz domains for $m>1$ and $d\ge 4$ remains a challenging open problem,
even for second-order elliptic systems with constant coefficients.
}
\end{remark}

\begin{remark}\label{remark-1.3}
{\rm 
Estimates in Theorems \ref{main-theorem-1} and \ref{main-theorem-2} 
were proved in \cite{KLS3} for the special case $\mathcal{L}=-\Delta$.
They were used to analyze the asymptotic behavior of the Dirichlet-to-Neumann maps
associated with elliptic systems with rapidly oscillating coefficients, arising in the theory of
homogenization.
We also point out that Theorems \ref{main-theorem-1} and \ref{main-theorem-2} continue to hold if
$g(x)=\big(g^{\alpha\beta}(x) \big)$ is an $m\times m$ matrix that commutes with $A(x)$; 
i.e. $g^{\alpha\beta} a_{ij}^{\beta\gamma}=a_{ij}^{\alpha\beta} g^{\beta\gamma}$.  
The proof is the same.
}
\end{remark}

Our approach to Theorems \ref{main-theorem-1} and \ref{main-theorem-2}, which is
similar to that used in \cite{KLS3} for $\mathcal{L}=-\Delta$,
 is based on Dahlberg's bilinear estimate
\begin{equation}\label{D-bilinear}
\aligned
& \left|\int_\Omega \nabla u \cdot v\, dx \right|\\
& \le C \left\{ \int_\Omega |\nabla u(x)|^2\, \delta(x)\, dx\right\}^{1/2}
\left\{ \int_\Omega |\nabla v (x)|^2\, \delta (x)\, dx +\int_{\partial\Omega} |(v)^*|^2\, d\sigma \right\}^{1/2},
\endaligned
\end{equation}
where $\delta (x)=\text{dist}(x, \partial\Omega)$,
$u=(u^\alpha) \in H^1(\Omega; \mathbb{R}^m)$ is a weak solution of
$\mathcal{L}(u)=0$ in $\Omega$, and
$v=(v_i^\alpha)\in H^1(\Omega; \mathbb{R}^{dm})$.
In (\ref{D-bilinear}) we have used $(v)^*$ to denote the nontangential maximal function 
of $v$.
The bilinear estimate was proved in \cite{Dahlberg-1986} for
harmonic functions $u$ in Lipschitz domains (see  related work in \cite{Dahlberg-1986-AJM, dkv-biharmonic}, where similar bilinear forms 
were used to solve $L^p$ Dirichlet problems in Lipschitz domains).
We also mention that the results in \cite{Dahlberg-1986} were extended in \cite{Hofmann-2008}
to a class of second-order elliptic operators in the upper half-space with time-independent complex coefficients.
In Section 2 we provide a relative simple proof
of (\ref{D-bilinear}), under the assumptions that $A$ satisfies (\ref{ellipticity}) and
$|\nabla A (x)|^2\, \delta (x)\, dx$ is a Carleson measure on $\Omega$.
The proof, which is probably known to experts in the area,
follows closely the basic argument in \cite{DKPV}, where
the equivalence in $L^p$ norms between the square function  and
the nontangential maximal function 
was established for solutions of higher-order elliptic systems
with constant coefficients in Lipschitz domains.
In Section 3 we use  a perturbation argument  to prove (\ref{D-bilinear})
for elliptic operators with H\"older continuous coefficients.

The proof of Theorem \ref{main-theorem-1} is given in Section 4, while
Section 5 is devoted to the proof of Theorem \ref{main-theorem-2}.
Let $h\in H^1(\Omega; \mathbb{R}^m)$ be a weak solution of
$\mathcal{L}^* (h)=-\text{div}(A^* (x)\nabla h)=0$ in $\Omega$.
The connection between the bilinear estimate (\ref{D-bilinear}) and the commutator 
$[\Lambda, g]$ is made through the following identity,
\begin{equation}\label{identity}
\aligned
& \int_{\partial\Omega} \big\{ \Lambda (gf)- g \Lambda (f)\big\} \cdot h\, d\sigma\\
&\qquad =\int_\Omega \frac{\partial v}{\partial x_i} \cdot u^\alpha \cdot a_{ji}^{\beta\alpha} \frac{\partial h^\beta}{\partial x_j}\, dx
-\int_\Omega \frac{\partial v}{\partial x_i} \cdot a_{ij}^{\alpha\beta} \frac{\partial u^\beta}{\partial x_j} \cdot h^\alpha\, dx,
\endaligned
\end{equation}
where $v\in H^1(\Omega)$ is any extension of $g$ to $\Omega$.
For the estimate (\ref{1.1}) we construct $v$ in such a way that
$\|\nabla v\|_{L^\infty(\Omega)}\le C \| g\|_{C^{0,1}(\partial\Omega)}$, and 
$d\nu =|\nabla^2 v (x)|^2\, \delta (x) \, dx$ is a Carleson measure on $\Omega$
with norm $\|\nu\|_{\mathcal{C}}$ less than $C \|g\|^2_{C^{0,1}(\partial\Omega)}$.
In the case of  (\ref{1.2}) we choose $v$ to be the harmonic extension of
$g$; i.e., $\Delta v=0$ in $\Omega$ and $v=g$ on $\partial\Omega$.

Finally, we remark that although we will not pursue the approach in this paper,
it seems possible, at least in the case of (\ref{1.1}),
 to establish the commutator estimates
for the Dirichlet-to-Neumann map, using the method of layer potentials.
Indeed, let
$$
\mathcal{S} (f)(x)=\int_{\partial\Omega} \Gamma (x,y) f(y)\, d\sigma(y)
$$
denote the single layer potential for $\mathcal{L}$, where $\Gamma (x,y)$
is the matrix of fundamental solutions for $\mathcal{L}$ in $\mathbb{R}^d$
(one may suitably modify $A$ outside $\Omega$ for the existence of $\Gamma (x,y)$).
The conormal derivative of $u=\mathcal{S} (f)$ is given by
$\big( (1/2) I +\mathcal{K}\big) f$, where $\mathcal{K}$ is a singular integral
operator on $\partial\Omega$.
It follows that
$$
\Lambda (f) = \big ( (1/2)I +\mathcal{K} \big) S^{-1} (f),
$$
where $S=\mathcal{S}|_{\partial\Omega}$.
This implies that
$$
\aligned
\big[ \Lambda, g\big] f
&=\big( (1/2) I +\mathcal{K} \big) \big[ {S}^{-1}, g\big] f +\big[\mathcal{K}, g\big] S^{-1} (f)\\
&=-\big( (1/2) I +\mathcal{K} \big) S^{-1} \big[ {S}, g\big] S^{-1} f +\big[\mathcal{K}, g\big] S^{-1} (f).
\endaligned
$$
Under the assumption that $A$ is elliptic and H\"older continuous,
it is known that for $1<p<\infty$,
the operator $\mathcal{K}$ is bounded on $L^p(\partial\Omega)$,
and $S$ is bounded from $L^p(\partial\Omega)$ to $W^{1, p}(\partial\Omega)$.
If, in addition, $A^*=A$, then $\big( (1/2) I +\mathcal{K}\big)$ is
invertible on $L^2_0(\partial\Omega)$, the subspace of $L^2$ functions with mean value zero.
Furthermore, if $d\ge 3$, $S$ is an isomorphism from $L^2(\partial\Omega)$ to $W^{1,2}(\partial\Omega)$
(see \cite{Verchota-1984, Mitrea-Taylor-2000, Kenig-Shen-2}).
As a result, we obtain
\begin{equation}\label{1.10-1}
\aligned
&\|\big[ \Lambda, g\big] f\|_{L^2(\partial\Omega)}\\
& \le C \left\{ \| \big[ S, g\big]\|_{W^{-1,2}(\partial\Omega) \to W^{1,2}(\partial\Omega)}
+\| \big[\mathcal{K}, g\big] \|_{W^{-1, 2}(\partial\Omega)
\to L^2(\partial\Omega)} \right\} \| f\|_{L^2(\partial\Omega)},
\endaligned
\end{equation}
where $W^{-1, 2}(\partial\Omega)=H^{-1} (\partial\Omega)$ 
is the dual of $W^{1,2}(\partial\Omega)=H^1(\partial\Omega)$.
This reduces the proof of the estimate (\ref{1.1}) to that of estimates
of commutators $[S, g]$ and $[\mathcal{K}, g]$.


\section{Dahlberg's bilinear estimate, Part I}

Let $\Omega$ be a bounded Lipschitz domain in $\mathbb{R}^d$ and
 $\delta (x)=\text{dist}(x, \partial\Omega)$.
 For a function $u$ in $\Omega$, the nontangential maximal function $(u)^*$ is defined by
\begin{equation}\label{NT}
(u)^* (Q)=\sup \big\{ | u(x)|: \ x\in \Omega \text{ and } |x-Q|< \alpha_0 \, \delta (x) \big\}
\end{equation}
for $Q\in \partial\Omega$, where $\alpha_0=\alpha_0 (\Omega)>1$ is a fixed large constant.

Let $\nu$ be a nonnegative measure on $\Omega$. We call $\nu$ a Carleson measure if
$$
\| \nu\|_{\mathcal{C}}: =\sup 
\left\{ \frac{\nu (B(Q, r)\cap \Omega)}{r^{d-1}}: Q\in \partial\Omega \text{ and } 0<r< \text{diam}(\Omega)\right\}
$$
is finite. The defining property of Carleson measures is that
\begin{equation}\label{Carleson-property}
\int_\Omega | u|\, d\nu \le C \|\nu\|_{\mathcal{C}} \int_{\partial\Omega} (u)^*\, d\sigma, 
\end{equation}
where $C$ depends only on the Lipschitz character of $\Omega$
(see e.g. \cite[Section 7.3] {Grafakos} for the case of $\mathbb{R}^d_+$).
The goal of this section is to establish the following theorem.

\begin{thm}\label{theorem-2.1}
Let $\Omega$ be a bounded Lipschitz domain in $\mathbb{R}^d$.
Let $\mathcal{L}=-\text{\rm div} (A(x)\nabla)$ with $A(x)$satisfying the ellipticity condition (\ref{ellipticity}).
Also assume that $A\in C^1(\Omega)$ and that
\begin{equation} \label{2.1-1}
d\nu=|\nabla A(x)|^2\, \delta (x)\, dx \text{ is a Carleson measure with norm $\| \nu\|_{\mathcal{C}}$
less than } C_0.
\end{equation}
Let $u\in H^1(\Omega; \mathbb{R}^m)$ be a weak solution of $\mathcal{L} (u)=0$ in $\Omega$.
Then, for any $v\in H^1(\Omega; \mathbb{R}^{d\times m})$,
\begin{equation}\label{2.1-2}
\aligned
& \left| \int_\Omega \nabla u \cdot v\, dx \right|\\
&\le C \left\{ \int_\Omega |\nabla u(x)|^2\,  \delta (x) \, dx \right\}^{1/2}
\left\{ \int_\Omega |\nabla v(x)|^2\, \delta (x) \, dx
+\int_{\partial\Omega} | (v)^*|^2\, d\sigma \right\}^{1/2},
\endaligned
\end{equation}
where $C$ depends only on $d$, $m$, $\mu$, $C_0$, and $\Omega$.
\end{thm}

\begin{proof}
By a partition of unity it suffices to estimate the integral of $\nabla u\cdot v$ over
$B(Q, r_0)\cap \Omega$ with $v\in H_0^1(B(Q, r_0))$, where
$Q\in \partial\Omega$ and $r_0>0$ is small.
Since the condition (\ref{2.1-1}) is translation and rotation invariant, we may assume that
$Q=0$ and
$$
B(0, C_0 r_0)\cap \Omega 
=B(0, C_0 r_0)\cap \big\{ (x^\prime, t)\in \mathbb{R}^d: x^\prime \in \mathbb{R}^{d-1}
\text{ and } t>\psi (x^\prime) \big\},
$$
where $\psi$ is a Lipschitz function on $\mathbb{R}^{d-1}$, $\psi (0)=0$, and
$C_0=10 (\|\nabla \psi\|_\infty +1)$.
Thus it is enough to establish the estimate  (\ref{2.1-2}) with $\Omega$ replaced by
$$
D=\big\{ (x^\prime, t)\in \mathbb{R}^d: x^\prime \in \mathbb{R}^{d-1}
\text{ and } t>\psi (x^\prime) \big\},
$$
assuming that $u$ is a weak solution of $\mathcal{L}(u)=0$ in $B(0, r_0)\cap D$
and $v\in H^1_0( B(0, r_0))$.
Using a special change of variables invented by C. Kenig and E. Stein,
we may further reduce the problem to the case of the upper half-space $\mathbb{R}^d_+
=\big\{ (y^\prime, s )\in \mathbb{R}^d: \ y^\prime \in \mathbb{R}^{d-1} \text{ and } s >0 \big\}$.
Indeed, consider the bi-Lipschitz map from $\mathbb{R}^d_+$ to $D$, defined by
$$
(y^\prime, s)\to \Phi (y^\prime, s)=\big(y^\prime, cs + \eta_s * \psi (y^\prime)\big)= \big(y^\prime, F(y^\prime, s)\big),
$$
where $\eta_s (y^\prime) =s^{1-d} \eta (y^\prime/s)$ is a smooth compactly supported bump
function and the constant $c=c(d, \|\nabla\psi\|_\infty)>0$ is so large that $\frac{\partial F}{\partial s} \ge 1$.
The key observations here are (1) $c\le |\nabla \Phi (y^\prime,s) |\le C$;  (2)
$|\nabla^2 \Phi (y^\prime ,s)|^2\, s \, dy^\prime ds$ is a Carleson measure on $\mathbb{R}^d_+$; and (3)
all constants depend only on $d$ and $\|\nabla\psi\|_\infty$ (see e.g. \cite{Dahlberg-1986,DKPV}).
It follows that the ellipticity condition (\ref{ellipticity}) and
the Carleson condition (\ref{2.1-1})
are preserved under the change of variables $(x^\prime, t)=\Phi(y^\prime, s)$.
 As a result, we only need to show that 
\begin{equation}\label{2.1-4}
\aligned
& \left| \int_{\mathbb{R}^d_+}  \nabla u \cdot v\, dx^\prime dt \right|\\
&\le C \left\{ \int_{\mathbb{R}^d_+}  |\nabla u(x^\prime, t)|^2\,  t \, dx^\prime dt \right\}^{1/2}
\left\{ \int_{\mathbb{R}^d_+}  |\nabla v(x^\prime, t)|^2\, t \, dx^\prime dt
+\int_{\mathbb{R}^{d-1}} | (v)^*|^2\, dx^\prime \right\}^{1/2},
\endaligned
\end{equation}
where $u$ is a weak solution of $\mathcal{L} (u) =0$ in $B(0, r_0)\cap \mathbb{R}^d_+$
and $v\in H^1_0( B(0,r_0))$. By suitably redefining $A$ outside $B(0, 3r_0)$, 
we may assume that $|\nabla A (x^\prime, t)|^2 \, t \, dx^\prime dt$
is a Carleson measure on $\mathbb{R}^d_+$ with a norm depending only on
$\mu, $, $C_0$, and $\|\nabla\psi\|_\infty$. 

We now proceed to prove the estimate (\ref{2.1-4}), using an approach found in \cite{DKPV}.
We begin by writing
\begin{equation}\label{2.1-6}
\aligned
\int_{\mathbb{R}^d_+} \nabla u\cdot v \, dx^\prime dt
&=\int_{\mathbb{R}^d_+} \nabla u\cdot v\cdot \frac{\partial t}{\partial t}\, dx^\prime dt\\
&=-\int_{\mathbb{R}^d_+} \frac{\partial }{\partial t} \nabla u \cdot v \cdot t \, dx^\prime dt
-\int_{\mathbb{R}^d_+} \nabla u\cdot \frac{\partial v}{\partial t} \, t \, dx^\prime dt,
\endaligned
\end{equation}
where we have used the integration by parts. By the Cauchy inequality, the second term in the right hand side of
(\ref{2.1-6}) is dominated in absolute value by the right hand side of (\ref{2.1-4}).
As for the first term,
it is easy to see that an integration by parts in $x^\prime$ and the  Cauchy inequality may be used to handle the integral of
$\frac{\partial}{\partial t} \left( \frac{\partial  u}{\partial x_i}\right) \cdot v \cdot  t$
for $i=1, \dots, d-1$.
As a result, it remains to bound the integral
\begin{equation}\label{2.1-8}
\int_{\mathbb{R}^d_+} \frac{\partial^2 u}{\partial t^2} \cdot v \cdot t \, dx^\prime dt,
\end{equation}
by the right hand side of (\ref{2.1-4}).
This will be done by using the assumption that $u$ is a solution of a second-order elliptic system.

Indeed, using $\mathcal{L}(u)=0$ in $B(0, r_0)\cap \mathbb{R}_+^d$, we may write
\begin{equation}\label{2.1-10}
a_{dd}^{\alpha\beta} \frac{\partial^2 u^\beta}{\partial t^2}
=\nabla_{x^\prime} \left\{ F^\alpha(x^\prime, t) \right\}
+G^\alpha (x^\prime,t),
\end{equation}
where $\nabla_{x^\prime} = (\frac{\partial }{\partial x_1}, \dots, \frac{\partial}{\partial x_{d-1}})$,
$F=(F^\alpha)$ and $G=(G^\alpha)$ satisfy $|F|\le C |\nabla u|$ and $|G|\le C |\nabla A|\, |\nabla u|$.
Note that the $m\times m$ matrix $\big(a_{dd}^{\alpha\beta} \big)$ is invertible by (\ref{ellipticity}),
and its inverse $E=\big( b^{\alpha\beta}\big)$ satisfies the same type of ellipticity and Carleson conditions as
$A$. This allows us to use the integration by parts in $x^\prime$ to obtain
\begin{equation} \label{2.1-12}
\aligned
\left|
\int_{\mathbb{R}^d_+} \frac{\partial^2 u}{\partial t^2} \cdot v \cdot t \, dx^\prime dt\right|
& \le C \int_{\mathbb{R}^{d}_+}
\bigg\{  |F|\,  |\nabla \{ E\cdot v \}|
+|E|\,  |\nabla A| \, |G|\,  |v|\bigg\} t\, dx^\prime dt\\
& \le C \int_{\mathbb{R}^d_+} 
\bigg\{ |\nabla u|\, |\nabla E|\, |v| + |\nabla u|\, |\nabla v|
+ |\nabla A|\, |\nabla u|\, |v| \bigg\} t\, dx^\prime dt.
\endaligned
\end{equation}
It follows  by the Cauchy inequality that the left hand side of (\ref{2.1-12})
is bounded by
\begin{equation}\label{2.1-13}
C \left\{ \int_{\mathbb{R}^d_+} |\nabla u|^2\, t \, dx^\prime dt\right\}^{1/2}
\left\{ \int_{\mathbb{R}^d_+} |\nabla v|^2\, t \, dx^\prime dt
+\int_{\mathbb{R}^d_+}  |v|^2 \big\{ |\nabla E|^2 +|\nabla A|^2\big\} t \, dx^\prime dt\right\}^{1/2}.
\end{equation}
Finally, since $\big\{ |\nabla E|^2 +|\nabla A|^2\big\} t\, dx^\prime dt$ is a Carleson measure on
$\mathbb{R}^d_+$, we see that (\ref{2.1-13}) is dominated by the right hand side of (\ref{2.1-4}).
This completes the proof.
\end{proof}


\section{Dahlberg's bilinear estimate, Part II}

In this section we show that the bilinear estimate (\ref{2.1-2}) holds
if $A$ is elliptic, symmetric, and H\"older continuous.
We mention that under these assumptions on $A$, the Dirichlet problem (\ref{Dirichlet})
is uniquely solvable for $f\in L^2(\partial\Omega; \mathbb{R}^m)$.
Moreover, the solution $u$ satisfies the nontangential maximal function and square function
estimates:
\begin{equation}\label{max-square}
\int_{\partial\Omega} |(u)^*|^2\, d\sigma
+\int_\Omega |\nabla u (x)|^2\, \delta (x)\, dx
\le C\, \int_{\partial\Omega} |f|^2\, d\sigma,
\end{equation}
where $C$ depends only on $A$ and $\Omega$ (see e.g. \cite{Mitrea-Taylor-2000, Kenig-Shen-2}).

\begin{thm}\label{theorem-3.1}
Let $\Omega$ be a bounded Lipschitz domain in $\mathbb{R}^d$.
Let $\mathcal{L}=-\text{\rm div} (A(x)\nabla )$
with $A(x)$ satisfying (\ref{ellipticity}). Also assume that $A^*=A$ and that
$A$ is H\"older continuous
in $\mathbb{R}^d$.
Let $u\in H^1(\Omega; \mathbb{R}^m)$ be a weak solution of $\mathcal{L}(u)=0$ in $\Omega$.
Then, for any $v\in H^1(\Omega; \mathbb{R}^{d\times m})$,
\begin{equation}\label{3.1-0}
\left|\int_\Omega \nabla u\cdot v\, dx\right|
\le C \| u\|_{L^2(\partial\Omega)} 
\left\{ \int_\Omega |\nabla v|^2\, \delta (x)\, dx
+\int_{\partial\Omega} |(v)^*|^2\, d\sigma \right\}^{1/2},
\end{equation}
where the constant $C$ depends at most on $A$ and $\Omega$.
\end{thm}

Theorem \ref{theorem-3.1} is a consequence of Theorem \ref{theorem-2.1} and the following lemma.

\begin{lemma}\label{lemma-3.1}
Assume that $\Omega$ and $A$ satisfy the same conditions as in Theorem \ref{theorem-3.1}.
Let $u\in H^1_0(\Omega; \mathbb{R}^m)$ be a weak solution of
$\mathcal{L}(u)=\text{\rm div} (f)$, where $f=(f_i^\alpha)\in L^2(\Omega; \mathbb{R}^{d\times m})$.
Then
\begin{equation}\label{3.1-2}
\int_\Omega |\nabla u(x)|^2 \, \delta (x) dx
\le C \int_\Omega |f(x)|^2\,  \delta (x) \, \phi (\delta (x))\, dx,
\end{equation}
where $ \phi (t) = \big\{ |\ln( t )|+1 \big\}^2$, and
$C$ depends only on $A$ and $\Omega$.
\end{lemma}

\begin{proof}
By dilation and translation we may assume that $\Omega\subset [-1/4, 1/4]^d$.
Without loss of generality we may further assume that $A(x)$ is periodic with respect to $\mathbb{Z}^d$.
This would allow us to use results in \cite{KLS2} obtained for second-order elliptic systems in divergence form 
with periodic coefficients.

Let $w=(w^\alpha)$, where
$$
w^\alpha (x)=-\int_\Omega \frac{\partial}{\partial y_i} \bigg\{ \Gamma^{\alpha\beta}
(x,y) \bigg\} f_i^\beta (y)\, dy,
$$
and $\big( \Gamma^{\alpha\beta} (x,y) \big)$ denotes the matrix of fundamental solutions
for $\mathcal{L}$ in $\mathbb{R}^d$, with pole at $y$.
It follows from \cite[section 8]{KLS2} that 
\begin{equation}\label{3.1-4}
\int_\Omega |\nabla w (x)|^2\, \delta (x)\, dx
\le C \int_\Omega |f(x)|^2\,  \delta (x)\,  \phi (\delta (x))\, dx,
\end{equation}
and 
\begin{equation}\label{3.1-6}
\int_{\partial\Omega} |w|^2\, d\sigma
\le C \int_\Omega |f(x)|^2\,  \delta (x) \, \phi (\delta (x))\, dx.
\end{equation}
Note that $\mathcal{L} (u-w)=0$ in $\Omega$ and $u-w=-w$ on $\partial\Omega$.
Using (\ref{max-square}), we obtain
$$
\int_{\Omega} |\nabla (u-w)|^2\, \delta (x)\, dx \le C \int_{\partial\Omega} |w|^2\, d\sigma.
$$
This, together with (\ref{3.1-6}), gives
\begin{equation}\label{3.1-8}
\int_\Omega |\nabla (u-w) |^2 \, \delta (x)\, dx
\le C \int_\Omega |f(x)|^2 \, \delta (x) \, \phi (\delta (x))\, dx.
\end{equation}
The desired estimate now follows from (\ref{3.1-4}) and (\ref{3.1-8}).
\end{proof}

\begin{remark}\label{remark-3.1}
{\rm
Since $\phi(t)\le C_\varepsilon t^{-\varepsilon}$ for $0<t<1$ and $\varep>0$, 
it follows from (\ref{3.1-2}) that
\begin{equation}\label{3.2-1}
\int_\Omega |\nabla u(x)|^2\, \delta (x)\, dx
\le C_\varep \int_\Omega |f(x)|^2\, \big[\delta (x) \big]^{1-\varep}\, dx.
\end{equation}
This, together with the energy estimate $\|\nabla u\|_{L^2(\Omega)}\le C \|f\|_{L^2(\Omega)}$,
gives
\begin{equation}\label{3.2-3}
\int_\Omega |\nabla u (x)|^2 \big[ \delta (x)\big]^{\alpha_2}\, dx
\le C\int_\Omega |f(x)|^2\big[\delta (x)\big]^{\alpha_1} \, dx,
\end{equation}
by complex interpolation, where $0\le \alpha_1<\alpha_2\le 1$.
Although the weighted norm inequality (\ref{3.2-3}) for $0<\alpha_1<\alpha_2<1$  is sufficient for our purpose,
it would be interesting to know if
(\ref{3.2-3}) holds for $0<\alpha_1=\alpha_2<1$.
}
\end{remark}

\begin{proof}[\bf Proof of Theorem \ref{theorem-3.1}]
Without loss of generality we may assume that
$\Omega\subset [-1/4, 1/4]^d$ and $A(x)$ is periodic with respect to $\mathbb{Z}^d$.
Suppose that $A$ is H\"older continuous of order $\eta$.
We construct a matrix $B(x)=\big( b_{ij}^{\alpha\beta} (x) \big)$, where $b_{ij}^{\alpha\beta}$
is the solution of the Dirichlet problem
$$
\Delta \big(b_{ij}^{\alpha\beta} \big) =0 \quad \text{ in } \Omega
\quad \text{ and } \quad b_{ij}^{\alpha\beta} =a_{ij}^{\alpha\beta} \quad \text{ on } \partial\Omega.
$$
It follows from the maximum principle that
 $B(x)$ satisfies the elliptic condition (\ref{ellipticity}) with the same $\mu$.
 Also note that $B\in C^\infty(\Omega)$, $B$ is H\"older continuous of order $\eta$
 in $\overline{\Omega}$,
\begin{equation}\label{3.3-1}
|\nabla B(x)|\le C \big[\delta(x)]^{\eta-1} \ \ \ \text{ for any } x\in \Omega,
\end{equation}
and 
\begin{equation}\label{3.3-3}
  |A(x)-B(x)|\le C \big[\delta (x)\big]^\eta \ \ \  \text{  for any  } x\in \Omega.
 \end{equation}
(here we have assumed that $0<\eta<\eta_0(\Omega)$ is sufficiently small).
In particular, the estimate (\ref{3.3-1}) implies that $|\nabla B(x)|^2\,  \delta (x)\, dx$
is a Carleson measure on $\Omega$.

Now, let $u\in H^1(\Omega;\mathbb{R}^m)$ be a weak solution of 
$\mathcal{L} (u)=0$ in $\Omega$ and $v\in H^1(\Omega; \mathbb{R}^{d\times m})$.
Write
\begin{equation}\label{3.3-5}
\int_\Omega \nabla u\cdot v\, dx
=\int_\Omega \nabla w \cdot v\, dx
+\int_\Omega \nabla (u-w)\cdot v\, dx,
\end{equation}
where $w\in H^1(\Omega; \mathbb{R}^m)$ is the solution of $\text{div} \big(B(x)\nabla w\big)=0$ in $\Omega$ and
$w=u$ on $\partial\Omega$.
By Theorem \ref{theorem-2.1} and (\ref{max-square}), the first term in the right hand side of 
(\ref{3.3-5}) is bounded by
$$
C \| u\|_{L^2(\partial\Omega)}
\left\{ \int_\Omega |\nabla v|^2\, \delta (x)\, dx +\int_{\partial\Omega} |(v)^*|^2\, d\sigma \right\}^{1/2}.
$$

The second term in the right hand side of (\ref{3.3-5}) will be handled by the estimate (\ref{3.2-3}).
By the Cauchy inequality we see that
\begin{equation}\label{3.3-7}
\aligned
\left|\int_\Omega \nabla (w-u)\cdot v\, dx \right|
& \le \left\{ \int_\Omega |\nabla (w-u)|^2 \, \big[\delta (x)\big]^{1-\eta}\, dx \right\}^{1/2}
\left\{ \int_\Omega |v|^2\, \big[\delta (x)\big]^{\eta-1}\, dx \right\}^{1/2}\\
&\le
C_\alpha \left\{ \int_\Omega |\nabla (w-u)|^2 \, \big[\delta (x)\big]^{1-\eta}\, dx \right\}^{1/2}
\| (v)^*\|_{L^2(\partial\Omega)},
\endaligned
\end{equation}
 where we have used the fact that $\big[\delta (x)\big]^{\eta-1} \, dx$ 
is a Carleson measure on $\Omega$ for the second inequality.
Note that $w-u\in H^1_0(\Omega; \mathbb{R}^m)$, and
$$
\text{div} \big(A\nabla (w-u)\big)
=\text{div} \big( A\nabla w\big)
=\text{div}\big( (A-B)\nabla w\big).
$$
It follows from the estimate  (\ref{3.2-3}) that
\begin{equation}\label{3.3-9}
\aligned
\int_\Omega
|\nabla (w-u)|^2 \big[\delta (x)\big]^{1-\eta}\, dx 
& \le C \int_\Omega |(A-B)\nabla w|^2 \big[\delta (x)\big]^{1-2\eta}\, dx\\
& \le C \int_\Omega |\nabla w|^2 \, \delta (x)\, dx\\
&\le C \int_{\partial\Omega} |u|^2\, d\sigma,
\endaligned
\end{equation}
where we have used (\ref{3.3-3}) for the second inequality and
(\ref{max-square}) for the third.
This, together with (\ref{3.3-7}), completes the proof.
\end{proof}


\section{Trilinear estimates and Proof of Theorem \ref{main-theorem-1}}

We begin with a lemma on extensions of Lipschitz functions.

\begin{lemma}\label{lemma-4.0}
Let $\Omega$ be a bounded Lipschitz domain.
Let $g\in C^{0,1} (\partial\Omega)$ be a scalar Lipschitz function on $\partial\Omega$.
Then there exists $v\in C(\overline{\Omega})\cap C^\infty(\Omega)$ such that
$v=g$ on $\partial\Omega$, $\|\nabla v\|_{L^\infty(\Omega)} \le C \|g\|_{C^{0,1}(\partial\Omega)}$,
and $d\nu=|\nabla^2 v (x)|^2\, \delta (x)\, dx$ is a Carleson measure on $\Omega$
with norm $\|\nu\|_{\mathcal{C}}\le C\|g\|_{C^{0,1}(\partial\Omega)}$,
where $C$ depends only on $\Omega$.
\end{lemma}

\begin{proof}
By a partition of unity we may assume that supp$(g)\subset B(Q, r_0)\cap\partial\Omega$
for some $Q\in \partial\Omega$ and some small $r_0>0$.
By translation and rotation we may assume that $Q=0$ and
$$
B(0, C_0r_0)\cap\Omega
=B(0, C_0r_0)\cap \big\{ (x^\prime, t)\in \mathbb{R}^d:\ 
x^\prime\in \mathbb{R}^{d-1} \text{ and } t>\psi(x^\prime)\big\},
$$
where $\psi$ is a Lipschitz function on $\mathbb{R}^{d-1}$, $\psi(0)=0$, and $C_0=10 (1+\|\nabla\psi\|_\infty)$.
Let $D=\big\{ (x^\prime, t)\in \mathbb{R}^d:
x^\prime\in \mathbb{R}^{d-1} \text{ and } t>\psi(x^\prime)\big\}$.
Recall that there exists a bi-Lipschitz map $\Phi: \mathbb{R}^{d}_+\to D$,
such that $c\le |\nabla \Phi(x^\prime, t)|\le C$, and
$|\nabla^2 \Phi(x^\prime,t)|^2\, t \, dx^\prime dt$ is a Carleson measure on $\mathbb{R}^d_+$.

Now, let $f(x^\prime)=g(\Phi^{-1}(x^\prime, 0))$.
Let $F$ be an extension of $f$ to $\mathbb{R}^d_+$ so that $|F|\le \| f\|_\infty$,
 $|\nabla F|\le C \|\nabla_{x^\prime} f\|_\infty$,
and $|\nabla^2 F(x^\prime,t)|^2 \, t\, dx^\prime dt$ is a Carleson measure with norm
less than $C\|\nabla_{x^\prime} f\|^2_\infty$.
It is not hard to verify that $v(x^\prime, t)=\varphi (x^\prime, t) F(\Phi^{-1}(x^\prime, t))$
satisfies all requirements in the lemma, if $\varphi\in C_0^\infty(B(0,C_0r_0))$
and $\varphi=1$ on $B(0, r_0)$.

Finally we remark that the extension $F$ may be given by $F(x^\prime, t)
=\eta_t* f (x^\prime)$, where $\eta_t (x^\prime)=t^{1-d} \eta(x^\prime/t)$, $\eta\in C_0^\infty(\mathbb{R}^{d-1})$,
and $\int_{\mathbb{R}^{d-1}} \eta =1$.
\end{proof}

Let $f=(f^\alpha)$ be an $\mathbb{R}^m$-valued  Lipschitz function on $\partial\Omega$
and $g$ a scalar Lipschitz function on $\partial\Omega$.
Let $u=(u^\alpha)\in H^1(\Omega; \mathbb{R}^m)$ be the weak solution of the Dirichlet problem,
\begin{equation}\label{4.0-1}
\mathcal{L}(u) =0 \quad \text{ in } \Omega \quad \text{ and } \quad
u=f \quad \text{ on } \partial \Omega.
\end{equation}
We will  use $v$ to denote the extension of $g$ to $\Omega$ given by Lemma \ref{lemma-4.0}.
Let $h\in H^1(\Omega; \mathbb{R}^m)$ be a weak solution of $\mathcal{L}^* (h)
=\text{div} \big(A^*(x)\nabla h\big)=0$ in $\Omega$.
The proof of Theorem \ref{main-theorem-1} relies on the following observation:
\begin{equation}\label{4.0-4}
\aligned
& \int_{\partial\Omega} \big\{ \Lambda (gf)- g \Lambda (f)\big\} \cdot h\, d\sigma
=\int_{\partial\Omega} gf \cdot \Lambda^* (h)\, d\sigma
-\int_{\partial\Omega} g \Lambda (f) \cdot h\, d\sigma\\
&=\int_{\partial\Omega} g\cdot  u^\alpha \cdot n_i a_{ji}^{\beta\alpha} \frac{\partial h^\beta}{\partial x_j}\, d\sigma
-\int_{\partial\Omega}
g \cdot n_i a_{ij}^{\alpha\beta} \frac{\partial u^\beta}{\partial x_j} \cdot h^\alpha\, d\sigma\\
&=\int_\Omega \frac{\partial v}{\partial x_i} \cdot u^\alpha \cdot a_{ji}^{\beta\alpha} \frac{\partial h^\beta}{\partial x_j}\, dx
-\int_\Omega \frac{\partial v}{\partial x_i} \cdot a_{ij}^{\alpha\beta} \frac{\partial u^\beta}{\partial x_j} \cdot h^\alpha\, dx,
\endaligned
\end{equation}
where we have used $\mathcal{L}(u)=0$ and $\mathcal{L}^* (h)=0$ in $\Omega$.

\begin{lemma}\label{lemma-4.1}
Let $\Omega$ be a bounded Lipschitz domain.
Assume that $A$ satisfies (\ref{ellipticity}) and that $d\nu=|\nabla A(x)|^2\, \delta(x)\, dx$ is a Carleson measure
on $\Omega$ with norm less  than $C_0$.
Let $f,g,h$, and $u, v$ be given above. Then
\begin{equation}\label{4.1-0}
\aligned
& \left|\int_{\partial\Omega} \big\{ \Lambda (gf)- g \Lambda (f)\big\} \cdot h\, d\sigma\right|\\
& \le C \|g\|_{C^{0,1}(\partial\Omega)}
\left\{ \int_\Omega |\nabla u|^2\, \delta (x)\, dx 
+\int_{\partial\Omega} |(u)^*|^2\, d\sigma\right\}^{1/2}\\
&\qquad\qquad \qquad \qquad\cdot \left\{ \int_\Omega |\nabla h|^2\, \delta (x)\, dx 
+\int_{\partial\Omega} |(h)^*|^2\, d\sigma\right\}^{1/2},
\endaligned
\end{equation}
where $C$ depends only on $d$, $m$, $\mu$, $C_0$, and $\Omega$.
\end{lemma}

\begin{proof}
The lemma follows  from (\ref{4.0-4}) and Theorem \ref{theorem-2.1}.
Indeed, since $\text{div} (A^*\nabla h)=0$ in $\Omega$,
by Theorem \ref{theorem-2.1}, the first term in the right hand side of (\ref{4.0-4}) is bounded
by
\begin{equation}\label{4.1-1}
C\left\{\int_\Omega |\nabla h|^2\, \delta (x)\,  dx\right\}^{1/2}
\left\{ \int_\Omega |\nabla F(x)|^2\, \delta (x) \, dx
+\int_{\partial\Omega} |(F)^*|^2\, d\sigma \right\}^{1/2}
\end{equation}
where $F=\big( \frac{\partial v}{\partial x_i} \cdot u^\alpha\cdot a_{ji}^{\beta\alpha}\big)$.
Note that $|F|\le C |\nabla v|\, |u|$ and
\begin{equation}\label{4.1-2}
|\nabla F|\le C \big\{ |\nabla^2 v| \, |u| +|\nabla v|\, |\nabla u| + |\nabla v|\,  |u|\,  |\nabla A|\big\}.
\end{equation}
It follows that
\begin{equation}\label{4.1-3}
\int_{\partial\Omega} |(F)^*|^2\, d\sigma
\le C \|g\|^2_{C^{0,1}(\partial\Omega)} \int_{\partial\Omega} |(u)^*|^2\, d\sigma,
\end{equation}
where we have used the estimate $\|\nabla v\|_{L^\infty(\Omega)} \le \|  g\|_{C^{0,1}(\partial\Omega)}$
in Lemma \ref{lemma-4.0}.
In view of (\ref{4.1-2}) we obtain
\begin{equation}\label{4.1-5}
\aligned
\int_{\partial\Omega}
|\nabla F|^2\, \delta (x)\, dx
& \le C\|g \|^2_{C^{0,1}(\partial\Omega)} \int_\Omega |\nabla u|^2\,  \delta (x)\, dx\\
& +C \int_\Omega |u|^2 |\nabla^2 v|^2 \, \delta (x)\, dx
+C \|g\|^2 _{C^{0,1}(\partial\Omega)} \int_\Omega |u|^2 |\nabla A|^2\, \delta(x)\, dx\\
&\le C \| g\|^2_{C^{0,1}(\partial\Omega)}
\left\{ \int_\Omega |\nabla u|^2\,  \delta (x)\, dx
+\int_{\partial\Omega} |(u)^*|^2\, d\sigma \right\},
\endaligned
\end{equation}
where we have used the fact that $|\nabla^2 v (x)|^2\,  \delta (x)\, dx$ is a Carleson measure on $\Omega$
with norm less than $C \|g \|^2_{C^{0,1}(\partial\Omega)}$ in Lemma \ref{lemma-4.0},
as well as the assumption that $|\nabla A(x)|^2 \, \delta (x)\, dx$ is a Carleson measure on $\Omega$.
This, together with (\ref{4.1-1}) and (\ref{4.1-3}), shows that the first term in the right hand side of
(\ref{4.0-4}) is dominated by the right hand side of (\ref{4.1-0}).

The second term in the right hand side of (\ref{4.0-4}) may be handled in the same fashion.
Since $\text{div}(A\nabla u)=0$ in $\Omega$, by Theorem \ref{theorem-2.1},  this term is bounded by
\begin{equation}\label{4.1-7}
C\left\{\int_\Omega |\nabla u(x) |^2\, \delta (x) dx\right\}^{1/2}
\left\{ \int_\Omega |\nabla G(x)|^2\, \delta (x) \, dx
+\int_{\partial\Omega} |(G)^*|^2\, d\sigma \right\}^{1/2},
\end{equation}
where $G=\big( \frac{\partial v}{\partial x_i} \cdot h^\alpha \cdot a_{ij}^{\alpha\beta} \big)$.
Note that $|G|\le C |\nabla v|\, |h|$ and
$$
|\nabla G|\le C \left\{ |\nabla^2v |\, |h| +|\nabla v|\, |\nabla h| +|\nabla v|\, |h| \, |\nabla A| \right\}.
$$
The rest of the argument is identical to that for the first term (with the roles of $u$ and $h$
switched). We omit the details.
\end{proof}

\begin{proof}[\bf Proof of Theorem \ref{main-theorem-1}]

Let $f=(f^\alpha)\in C^{0,1}(\partial\Omega; \mathbb{R}^m)$ and $g\in C^{0,1}(\partial\Omega)$.
Let $h\in H^1(\Omega;\mathbb{R}^m)$ be a weak solution of $\mathcal{L}(h)=0$ in $\Omega$.
It follows from (\ref{4.0-4}) that
\begin{equation}\label{4.2-1}
\aligned
&\left| \int_{\partial\Omega} \big\{ \Lambda (gf)- g\Lambda (f)\big\} \cdot h\, d\sigma \right|\\
&\ \ \ \ \le
\left|\int_\Omega \frac{\partial v}{\partial x_i} \cdot u^\alpha \cdot 
a_{ji}^{\beta\alpha} \frac{\partial h^\beta}{\partial x_j}\, dx\right|
+\left|\int_\Omega \frac{\partial v}{\partial x_i} \cdot a_{ij}^{\alpha\beta} 
\frac{\partial u^\beta}{\partial x_j} \cdot h^\alpha\, dx\right| =I_1 +I_2,
\endaligned
\end{equation}
where $v$ is the extension of $g$ given by Lemma \ref{lemma-4.0}.
We will show that
\begin{equation}\label{4.2-2}
I_1 +I_2 \le C \| g\|_{C^{0,1}(\partial\Omega)} \| f\|_{L^2(\partial\Omega)} \| h\|_{L^2(\partial\Omega)},
\end{equation}
which yields the desired estimate  by duality.

To estimate $I_1$, we construct an elliptic matrix $B(x)=(b_{ij}^{\alpha\beta})$,
as in the proof of Theorem \ref{theorem-3.1}. Note that 
\begin{equation}\label{4.2-3}
I_1
\le \left|\int_\Omega \frac{\partial v}{\partial x_i} \cdot u^\alpha \cdot 
b_{ji}^{\beta\alpha} \frac{\partial h^\beta}{\partial x_j}\, dx\right|
+
\left|\int_\Omega \frac{\partial v}{\partial x_i} \cdot u^\alpha \cdot 
\big( a_{ji}^{\beta\alpha} -b_{ij}^{\alpha\beta}\big)
\frac{\partial h^\beta}{\partial x_j}\, dx\right|=I_{11} +I_{12}.
\end{equation}
Since $|\nabla^2 B(x)|^2\, \delta (x)\, dx$ is a Carleson measure on $\Omega$,
we may use Theorem \ref{theorem-3.1} to show that
\begin{equation}\label{4.2-5}
I_{11}
\le C \| h\|_{L^2(\partial\Omega)} \| g\|_{C^{0,1}(\partial\Omega)}
\left\{ \int_{\Omega} |\nabla u|^2\, \delta (x)\, dx +\int_{\partial\Omega} |(u)^*|^2\, d\sigma \right\}^{1/2}
\end{equation}
(see the proof of Lemma \ref{lemma-4.1}). This, together with (\ref{max-square}), gives
\begin{equation}
|I_{11}|\le C \| h\|_{L^2(\partial\Omega)} \| g\|_{C^{0,1}(\partial\Omega)} \| f\|_{L^2(\partial\Omega)}.
\end{equation}

To estimate $I_{12}$, we observe that
$$
\aligned
I_{12}
&\le 
\int_\Omega |\nabla v|\, |u|\, |A-B|\, |\nabla h|\, dx \\
& \le C \| g\|_{C^{0,1}(\partial\Omega)}
\left\{ \int_\Omega |\nabla h|^2\, \delta (x)\, dx \right\}^{1/2}
\left\{\int_\Omega |u|^2 |A-B|^2 \, \big[\delta(x)\big]^{-1}\, dx\right\}^{1/2}\\
&\le C \| g\|_{C^{0,1}(\partial\Omega)} \| h\|_{L^2(\partial\Omega)}
\left\{\int_\Omega |u|^2 |A-B|^2 \, \big[\delta(x)\big]^{-1}\, dx\right\}^{1/2},
\endaligned
$$
where we have used the Cauchy inequality for the second inequality and (\ref{max-square}) for the third.
Since $|A(x)-B(x)|\le C \big[\delta(x)\big]^\eta$ for some $\eta>0$,
$|A-B|^2 \big[\delta (x)\big]^{-1}\, dx$ is a Carleson measure on $\Omega$.
It follows that
$$
|I_{12}|\le C \| g\|_{C^{0,1}(\partial\Omega)} \| h\|_{L^2(\partial\Omega)} \| f\|_{L^2(\partial\Omega)}.
$$

Finally, we point out that the term $I_2$ may be handled by the same manner as $I_1$, with the roles
of $u$ and $h$ switched. We omit the details.
\end{proof}


\section{Proof of Theorem \ref{main-theorem-2}}

Throughout this section we will assume that $A$ is elliptic, symmetric, and H\"older continuous.
Let $f\in C^{0, 1}(\partial\Omega; \mathbb{R}^m)$ and $u\in H^1(\Omega; \mathbb{R}^m)$ be
the weak solution  of $\mathcal{L}(u)=0$ in $\Omega$ with boundary data $u=f$ on $\partial\Omega$.
Let $h\in H^1(\Omega;\mathbb{R}^m)$ be a weak solution of $\mathcal{L} (h)=0$ in $\Omega$.
In this section we shall use $v$ to denote the harmonic extension of $g$ to $\Omega$; i.e.,
$\Delta v =0$ in $\Omega$ and $v=g$ on $\partial\Omega$.
It was proved in \cite{Dahlberg-1980, Jerison-Kenig-1980} that
\begin{equation}\label{5.0-0}
\int_{\partial\Omega} |(\nabla v)^*|^2\, d\sigma
+\int_\Omega |\nabla^2  v|^2\, \delta (x)\, dx 
\le C \, \| g\|_{H^1(\partial\Omega)}^2,
\end{equation}
where $C$ depends only on the Lipschitz character of $\Omega$. 
Our goal is to show that
\begin{equation}\label{5.0-1}
\left| \int_{\partial\Omega}
\big\{ \Lambda (gf) -g \Lambda (f) \big\}\cdot h\, d\sigma\right|
\le C\,  \| g\|_{H^1(\partial\Omega)}
\| u\|_{L^\infty(\Omega)} \| h\|_{L^2(\partial\Omega)},
\end{equation}
which yields the desired estimate in Theorem \ref{main-theorem-2} by duality.
We shall assume that $\|u\|_{L^\infty(\Omega)}$ is finite.
Using the localized square function estimate, one may show that $|\nabla u(x)|^2\, \delta (x)\, dx$ is a Carleson
measure on $\Omega$ with norm less than $C\|u\|_{L^\infty(\Omega)}^2$, where $C$ depends only
on $A$ and $\Omega$.

To prove (\ref{5.0-1}) we begin by recalling from  (\ref{4.0-4}) that
\begin{equation}\label{5.0-3}
\aligned
 \int_{\partial\Omega} \big\{ \Lambda (gf)- g \Lambda (f)\big\} \cdot h\, d\sigma
&=\int_\Omega \frac{\partial v}{\partial x_i} \cdot u^\alpha \cdot a_{ji}^{\beta\alpha} \frac{\partial h^\beta}{\partial x_j}\, dx
-\int_\Omega \frac{\partial v}{\partial x_i} \cdot a_{ij}^{\alpha\beta} \frac{\partial u^\beta}{\partial x_j} \cdot h^\alpha\, dx\\
&=I - J.
\endaligned
\end{equation}
To estimate $I$, we write
\begin{equation}\label{5.0-5}
\aligned
I &= \int_\Omega \frac{\partial v}{\partial x_i} \cdot u^\alpha \cdot b_{ji}^{\beta\alpha} \frac{\partial h^\beta}{\partial x_j}\, dx
+ \int_\Omega \frac{\partial v}{\partial x_i} \cdot u^\alpha \cdot \big( a_{ji}^{\beta\alpha}
-b_{ji}^{\beta\alpha} \big) \frac{\partial h^\beta}{\partial x_j}\, dx\\
&= I_1 +I_2,
\endaligned
\end{equation}
where $B=(b_{ij}^{\alpha\beta})$ is the elliptic matrix constructed in the proof of Theorem \ref{theorem-3.1}.
It follows from Theorem \ref{theorem-3.1} that
\begin{equation}\label{5.0-7}
|I_1|\le C \| h\|_{L^2(\partial\Omega)}
\left\{\int_{\Omega} |\nabla F|^2\, \delta (x)\, dx
+\int_{\partial\Omega} |(F)^*|^2\, d\sigma \right\}^{1/2},
\end{equation}
where $F=\big( \frac{\partial v}{\partial x_i}\cdot u^\alpha \cdot b_{ji}^{\beta\alpha}\big)$.
Using $(F)^* \le C\,  \| u\|_{L^\infty(\Omega)} (\nabla v)^*$, we obtain
\begin{equation}\label{5.0-9}
\int_{\partial\Omega} |(F)^*|^2\, d\sigma
\le C \, \| u\|^2_{L^\infty(\Omega)} \int_{\partial \Omega} |(\nabla v)^*|^2\, d\sigma
\le C \, \| u\|^2_{L^\infty(\Omega)} \| g\|^2_{H^1(\partial\Omega)},
\end{equation}
where we have used (\ref{5.0-0}) for the second inequality. Since
$$
|\nabla F|\le C \big\{ |\nabla^2 v|\, |u| +|\nabla v|\, |\nabla u| +|\nabla v|\, |u|\, |\nabla B|\big\},
$$
we see that
\begin{equation}\label{5.0-11}
\aligned
\int_{\Omega} |\nabla F|^2\, \delta (x)\, dx
& \le C \|u\|^2_{L^\infty(\Omega)}\int_\Omega |\nabla^2 v|^2\, \delta (x)\, dx
+C \int_\Omega |\nabla v|^2 |\nabla u|^2 \, \delta (x) \, dx\\
& \qquad\qquad+C \| u\|^2_{L^\infty(\Omega)}
\int_\Omega |\nabla v|^2 |\nabla B|^2\, \delta (x)\, dx\\
&\le C\| u\|^2_{L^\infty(\Omega)}
\left\{ \int_\Omega |\nabla^2 v|^2\, \delta (x)\, dx
+\int_{\partial\Omega} |(\nabla v)^*|^2\, d\sigma \right\}\\
&\le C \| u\|^2_{L^\infty(\Omega)}
\|g\|^2_{H^1(\partial\Omega)},
\endaligned
\end{equation}
where we have used the fact that $|\nabla u(x)|^2\delta (x)\, dx$ is a Carleson  measure on $\Omega$
with norm less than $C\|u\|^2_{L^\infty(\Omega)}$, and that
$|\nabla B(x)|^2\, \delta (x)\, dx$ is also a Carleson measure on $\Omega$. This, together with (\ref{5.0-7}) and
(\ref{5.0-9}),
gives
\begin{equation}\label{5.0-13}
|I_1| \le C \| g\|_{H^1(\partial\Omega)} \| u\|_{L^\infty(\Omega)} \| h\|_{L^2(\partial\Omega)}.
\end{equation}

To bound $I_2$, we use
$$
|I_2|
\le C \|u\|_{L^\infty(\Omega)}
\int_\Omega |\nabla v|\, |A-B|\, |\nabla h|\, dx.
$$
It follows by the Cauchy inequality that
$$
\aligned
|I_2| &\le C \| u\|_{L^\infty(\Omega)}
\left\{ \int_\Omega |\nabla v|^2 |A-B|^2 \big[\delta (x)\big]^{-1}\, dx \right\}^{1/2}
\left\{ \int_\Omega |\nabla h|^2\, \delta (x)\, dx\right\}^{1/2}\\
&\le C \|u\|_{L^\infty(\Omega)}
\| (\nabla v)^*\|_{L^2(\partial\Omega)} \| h\|_{L^2(\partial\Omega)}\\
&\le C \|u\|_{L^\infty(\Omega)}
\| g\|_{H^1(\partial\Omega)} \| h\|_{L^2(\partial\Omega)},
\endaligned
$$
where we have used the fact that $ |A-B|^2 \big[\delta (x)\big]^{-1}\, dx$ is a Carleson measure
on $\Omega$. The estimate of $I=I_1+I_2$ is now complete. 

Next, we turn to the estimate of $J$ in (\ref{5.0-3}).
Write
\begin{equation}\label{5.0-15}
\aligned
J&=\int_\Omega \frac{\partial v}{\partial x_i}\cdot b_{ij}^{\alpha\beta} \frac{\partial u^\beta}{\partial x_j}
\cdot h^\alpha\, dx
+
\int_\Omega \frac{\partial v}{\partial x_i}\cdot \big( a_{ij}^{\alpha\beta}-
b_{ij}^{\alpha\beta}\big) \frac{\partial u^\beta}{\partial x_j}
\cdot h^\alpha\, dx\\
&=J_1 +J_2.
\endaligned
\end{equation}
For $J_1$, we use the integration by parts to obtain
\begin{equation}\label{5.0-17}
\aligned
J_1& =-\int_\Omega \frac{\partial}{\partial x_j}\left\{ b_{ij}^{\alpha\beta} \frac{\partial v}{\partial x_i}\right\}
\cdot u^\beta\cdot h^\alpha\, dx
-\int_\Omega \frac{\partial v}{\partial x_i} \cdot b_{ij}^{\alpha\beta} u^\beta\cdot \frac{\partial h^\alpha}{\partial x_i}\, dx\\
& \qquad\qquad \qquad+\int_{\partial\Omega} n_j \frac{\partial v}{\partial x_i}
\cdot b_{ij}^{\alpha\beta} u^\beta \cdot h^\alpha\, d\sigma\\
& =J_{11}+J_{12} +J_{13}.
\endaligned
\end{equation}
Note that $J_{12}=-I_1$ and  by the Cauchy inequality,
$$
|J_{13}|\le C \| u\|_{L^\infty(\Omega)} \| \nabla v\|_{L^2(\partial \Omega)}
\| h\|_{L^2(\partial\Omega)}
\le C\, \| u\|_{L^\infty(\Omega)} \| g\|_{H^1(\partial\Omega)} \| h\|_{L^2(\partial\Omega)}.
$$
Also,
\begin{equation}\label{5.0-19}
|J_{11}|
\le \left|\int_\Omega b_{ij}^{\alpha\beta} \frac{\partial^2 v}{\partial x_i\partial x_j} \cdot u^\beta\cdot h^\alpha\, dx\right|
+\left|
\int_\Omega \frac{\partial b_{ij}^{\alpha\beta}}{\partial x_j} \cdot \frac{\partial v}{\partial x_i}\cdot u^\beta
\cdot h^\alpha\, dx \right|.
\end{equation}
Since $\nabla v$ is harmonic in $\Omega$, we may use Theorem \ref{theorem-3.1} to bound
 the first term in the right hand side of (\ref{5.0-19}) by
\begin{equation}\label{5.0-20}
C \| \nabla v\|_{L^2(\partial\Omega)}
\left\{ \int_\Omega |\nabla G|^2\, \delta (x)\, dx +\int_{\partial\Omega} |(G)^*|^2\, d\sigma \right\}^{1/2},
\end{equation}
where $G =\big( b_{ij}^{\alpha\beta} u^\beta h^\alpha\big)$.
As in the case of $I_1$, the term in (\ref{5.0-20}) is dominated by the right hand side of (\ref{5.0-13}),
using the fact that $|\nabla B(x)|^2\, \delta(x)\, dx $ and $|\nabla u(x)|^2\, \delta(x)\, dx$
are Carleson measures on $\Omega$.
By the Cauchy inequality the second term in the right hand side of (\ref{5.0-19})
is bounded by
$$
\aligned
& C \| u\|_{L^\infty(\Omega)}
\left\{ \int_\Omega |\nabla v|^2 |\nabla B|^2 \big[\delta (x)\big]^{1-\eta}\, dx \right\}^{1/2}
\left\{\int_\Omega |h|^2\, \big[\delta(x)\big]^{\eta-1}\, dx\right\}^{1/2} \\
& \le C \|u\|_{L^\infty(\Omega)} \| g\|_{H^1(\partial\Omega)} \| h\|_{L^2(\partial\Omega)},
\endaligned
$$ 
where we have used the estimate $|\nabla B(x)|\le C \big[\delta (x)\big]^{\eta-1}$
as well as the fact that $\big[\delta(x)\big]^{\eta-1}$ is a Carleson measure on $\Omega$.

Finally, it remains to estimate
$$
J_2=
\int_\Omega \frac{\partial v}{\partial x_i}\cdot \big( a_{ij}^{\alpha\beta}-
b_{ij}^{\alpha\beta}\big) \frac{\partial u^\beta}{\partial x_j}
\cdot h^\alpha\, dx.
$$
Recall that $|A(x)-B(x)|\le C \big[\delta(x)\big]^\eta$.
By the Cauchy inequality we see that
$$
\aligned
J_2 & \le
\left\{ \int_\Omega |\nabla v|^2 |A-B|^2 |\nabla u|^2 \big[\delta(x)\big]^{1-2\eta}\, dx \right\}^{1/2}
\left\{\int_\Omega | h|^2 \big[\delta (x)\big]^{2\eta -1}\, dx \right\}^{1/2}\\
&\le C \| h\|_{L^2(\partial\Omega)} \left\{ \int_\Omega |\nabla v|^2 |\nabla u|^2 \delta (x)\, dx \right\}^{1/2}\\
&\le C \| h\|_{L^2(\Omega)} \| u\|_{L^\infty(\Omega)} \| g\|_{H^1(\partial\Omega)},
\endaligned
$$
where we have used the fact that $|\nabla u(x)|^2\, \delta (x)\, dx$ is
a Carleson measure with norm less than $C \|u\|_{L^\infty(\Omega)}$.
This completes the proof of Theorem \ref{main-theorem-2}.

\bibliography{S32.bbl}

\bigskip

\small
\noindent\textsc{Department of Mathematics, 
University of Kentucky, Lexington, KY 40506}\\
\emph{E-mail}: \texttt{zshen2@uky.edu} \\


\end{document}